\documentclass[11 pt]{article}

\usepackage[latin1]{inputenc}
\usepackage{amsmath,amsfonts,stmaryrd,hyperref,amsthm, enumitem, graphicx}

\usepackage{fullpage}
\usepackage{authblk}

\newtheorem{thm}{Theorem}
\newtheorem{lemma}[thm]{Lemma}
\newtheorem{prop}[thm]{Proposition}
\newtheorem{cor}[thm]{Corollary}
\theoremstyle{definition}
\newtheorem{definition}{Definition}
\newtheorem{rem}{Remark}

\newcommand{\Aut}{\text{Aut}}
\renewcommand{\deg}{\text{deg}}

\newcommand{\cE}{\mathcal{E}}
\newcommand{\cG}{\mathcal{G}}
\newcommand{\cT}{\mathcal{T}}
\newcommand{\cU}{\mathcal{U}}

\newcommand{\subscript}[2]{$#1 _ #2$}

\title{Connectivity in bridge-addable graph classes: the McDiarmid-Steger-Welsh conjecture}

\author[*]{Guillaume Chapuy%
\thanks{
Support from \emph{Agence Nationale de la Recherche}, grant number ANR~12-JS02-001-01 ``Cartaplus'', and from the City of Paris, grant ``\'Emergences Paris 2013, Combinatoire \`a Paris''.
This work was done while G.C. was visiting McGill University, School of Computer Science -- many thanks to Luc Devroye for his hospitality. 
Email:~{\tt guillaume.chapuy@liafa.univ-paris-diderot.fr}.
}
}
\author[**]{Guillem Perarnau%
\thanks{
Email:~{\tt guillem.perarnaullobet@mcgill.ca}.
}
}
\affil[*]{\small\it {\sc IRIF, UMR CNRS 8243},
Universit\'e Paris-Diderot,
France.

\small\it {\sc CRM, UMI CNRS 3457},
Universit\'e de Montr\'eal,
Canada.}
\affil[**]{\small\it School of Computer Science, McGill University, Montr\'eal, QC, Canada.
}

\begin{document}
\maketitle
\begin{abstract}
A class of graphs is \emph{bridge-addable} if given a graph $G$ in the class, any graph obtained by adding an edge between two connected components of $G$ is also in the class. We prove a conjecture of McDiarmid, Steger, and Welsh, that says that if
$\mathcal{G}_n$ is any bridge-addable class of graphs on $n$ vertices, and $G_n$ is taken uniformly at random from $\mathcal{G}_n$, then $G_n$ is connected with probability at least $e^{-\frac{1}{2}} + o(1)$, when $n$ tends to infinity. This lower bound is asymptotically best possible since it is reached for forests.


Our proof uses a ``local double counting'' strategy that may be of independent interest, and that enables us to compare the size of two sets of combinatorial objects by solving a related multivariate optimization problem. In our case, the optimization problem deals with partition functions of trees relative to a supermultiplicative functional.
 
\end{abstract}

\section{Introduction, notation, main result}

In this paper, unless otherwise stated, all graphs are finite, simple and with vertex set $\{1,\dots,n\}$ for some $n\geq 1$. Following~\cite{MCSW}, we say that a family $\mathcal{G}$ of  graphs is \emph{bridge-addable}  if the following is true:
\begin{quote}
If $G$ is a graph from $\mathcal{G}$, and if $e$ is an edge not in $G$ whose endpoints belong to two different connected components of $G$, then the graph $G\cup\{e\}$ obtained by adding $e$ to $G$ is in $\mathcal{G}$.
\end{quote}

\noindent
The notion of bridge-addability was motivated by the study of connectivity in random planar graphs (since the class of planar graphs is clearly bridge-addable). Other examples of bridge-addable classes include forests, triangle-free graphs, graphs having a perfect matching, or any minor closed class of graphs whose excluded minors are 2-connected. See~\cite{MCSW, ABMCR} for even more examples.

In~\cite{MCSWplanar, MCSW}, the authors investigate the properties of a graph taken uniformly at random from a bridge-addable class. In particular they show the following:
\begin{prop}[\cite{MCSWplanar}]
For every $\epsilon>0$, there exists an $n_0$ such that for every $n\geq n_0$
and any bridge-addable class $\mathcal{G}_n$ of graphs with $n$ vertices, we have
\begin{align}\label{eq:weak}
\Pr \left( G_n \mbox{ is connected} \right) \geq (1-\epsilon)e^{-1}\;,
\end{align}
where $G_n$ is chosen uniformly at random from $\mathcal{G}_n$.
\end{prop}
The authors of~\cite{MCSW} conjectured that the constant $e^{-1}$ in~\eqref{eq:weak} can be improved to $e^{-1/2}$. If true, this would be best possible, since it is proved in~\cite{Renyi} that if $F_n$ is a uniform random forest on $n$ vertices, then one has when $n$ tends to infinity:
$$
\Pr\left(F_n \mbox{ is connected}\right) \longrightarrow e^{-1/2}.
$$
 \noindent The first partial result towards the proof of the conjecture was obtained by Balister, Bollob\'as, and Gerke~\cite{BBG} who improved the constant in~\eqref{eq:weak} from $e^{-1}$ to $e^{-0.7983}$.
Norin, in an unpublished draft~\cite{Norine}, improves it to $e^{-2/3}$. Until the present paper these were, as far as we know, the best results under general hypotheses. 

Under the (much) stronger hypothesis that the class is also \emph{bridge-alterable} (i.e. that the class is \emph{also} stable by bridge deletion), Addario-Berry, McDiarmid and Reed~\cite{ABMCR}, and Kang and Panagiotou~\cite{KP} independently improved the constant to the general conjectured value $e^{-1/2}$. Both proofs use the fact that graphs from bridge-alterable classes can be encoded by weighted forests, so that the problem reduces to estimating the connectivity in a random weighted forest.
 Unfortunately, this encoding fails for general bridge-addable classes, so these proof techniques do not apply to the general case. Note also that many bridge-addable classes are \emph{not} bridge-alterable, for example graphs that admit a perfect matching, graphs that have a component of large size, or  graphs in any bridge-addable class that contain a  given subgraph (for example, planar graphs containing a  path of given length).

In this paper, we prove the general conjecture, \textit{i.e.} we improve the constant in \eqref{eq:weak} to $e^{-1/2}$ using only the hypothesis of bridge-addability. 
Our main result is the following:
\begin{thm}\label{thm:main}
The McDiarmid-Steger-Welsh conjecture is true: For every $\epsilon>0$, there exists an $n_0$ such that for every $n\geq n_0$ 
and any bridge-addable class $\mathcal{G}_n$ of graphs with $n$ vertices 
we have
\begin{align*}
\Pr \left( G_n \mbox{ is connected} \right) \geq (1-\epsilon)e^{-1/2}\;,
\end{align*}
where $G_n$ is chosen uniformly at random from $\mathcal{G}_n$.
\end{thm}

We  believe that, beyond the own interest of the conjecture, our proof is interesting for several other reasons. First, the resolution of this problem promotes further this ``alternative approach'' to random graphs with constraints, and raises the question of finding other class properties, beyond bridge-addability, that give rise to interesting and non-trivial predictable asymptotic behaviour. For example McDiarmid and Scott~\cite{McDiarmidScott} have studied \emph{block-classes} (where the class property is stability under taking a two-connected component) and Balister, Bollob\'{a}s and Gerke~\cite{BBG} have studied the property of \emph{$2$-addability}. One can expect that in years to come both more examples and a greater unification of the techniques will be possible. Second, our proof is based on a combinatorial technique that we call ``local double counting'' that is interesting on its own and that we believe may be useful in many other situations. We give a quick description of this technique in the next section: roughly speaking, it mixes a classical double counting technique already used in previous works on this topic with a careful analysis of the \emph{local structure} of the objects under study. 
The idea of taking advantage of the local structure of the graphs to improve the double-counting argument has already been used in the papers~\cite{BBG,Norine} via a \emph{weighted} version of the double-counting setup. As far as we know, the technique we introduce here is of a different nature. In particular, to prevent any confusion, we mention that the word \emph{weight} in this paper is used as in statistical mechanics to refer to the pondering used in the definition of some partition functions that appear in an optimization problem inherited from our double-counting arguments.  However here the underlying double-counting is  \emph{unweighted} in nature, and the efficiency of our method comes from the fact it is performed \emph{locally} (see Section~\ref{sec:discussion}).
	
Finally, our proof gives more information on random graphs from bridge-addable classes than what is contained in Theorem~\ref{thm:main}. We are able to quantify the probability of having any small number of connected components:
\begin{thm}\label{thm:main2}
For every $\epsilon>0$ and for every $k \geq 0$, there exists an $n_0$ such that for $n\geq n_0$ one has:
\begin{eqnarray*}
\Pr \left( G_n \mbox{ has at most $k\!+\!1$ connected components} \right)
 \geq
\Pr \left( \mbox{Poisson}\left(\frac{1}{2}\right)  \leq k \right) - \epsilon.
\end{eqnarray*}
\end{thm}
The proof is a simple extension of the proof of~Theorem~\ref{thm:main} (that corresponds to $k=0$).
\medskip

In the rest of the paper, we fix an integer $n$ and a bridge-addable class $\cG_n$ such that all graphs in $\cG_n$ have $n$ vertices.
We let $\cG_n^{(i)}$ be the graphs in $\cG_n$ having exactly $i$ connected components, and we use the shortcut notations $\mathcal{A}_n:=\cG^{(1)}_n$ and $\mathcal{B}_n:=\cG^{(2)}_n$.
The main ingredient in the proof of Theorem~\ref{thm:main} is the following proposition:
\begin{prop}\label{prop:AvsB}
For all $\epsilon'>0$ there exists $n_0$ such that for $n\geq n_0$ one has:
\begin{eqnarray}\label{eq:AvsB}
|\mathcal{B}_n|
\leq \left(\frac{1}{2}+\epsilon' \right) |\mathcal{A}_n|.
\end{eqnarray}
\end{prop}

\medskip
\noindent{\bf Structure of the paper}.
We start in Section~\ref{sec:discussion} with a high level discussion on the ``local double counting'' strategy that we use to prove our main result and that may be of independent interest -- the reader only interested in the proofs can skip this discussion.
In Section~\ref{sec:admit}, we prove Theorems~\ref{thm:main} and~\ref{thm:main2} admitting Proposition~\ref{prop:AvsB}. The proof of Proposition~\ref{prop:AvsB} occupies the rest of the paper.  In Section~\ref{sec:local}, we define the local parametrization of our graph classes (Section~\ref{subsec:localparams}), we state the main combinatorial double counting argument (Section~\ref{subsec:localcounting}), and we use it to obtain a local bound (Corollary~\ref{cor:sumBound} in Section~\ref{sec:weights}) on a functional of some ratios relating the size of $\mathcal{A}_n$ to the size of certain subsets of $\mathcal{B}_n$. This functional takes the form of a truncated partition function of rooted trees carrying some supermultiplicative weights. In order to use this bound, we study partition functions of trees in  Section~\ref{sec:partitionfunctions}. In Section~\ref{subsec:defspartition}, we first study  untruncated partition functions, and we relate the rooted and unrooted case via an analogue of the dissymmetry theorem (Lemma~\ref{lemma:disymmetry}). In Section~\ref{subsec:optimization}, we transfer the results of Section~\ref{subsec:defspartition} to the setting of truncated partition functions (Proposition~\ref{prop:obj}). Finally, in Section~\ref{sec:finishproof}, we finish the proof of Proposition~\ref{prop:AvsB}. In Section~\ref{subsec:boxing}, we show that we can choose a good local partitioning of our graphs classes,
that enables us to apply the results of Sections~\ref{sec:local} and~\ref{sec:partitionfunctions} to conclude the proof 
in Section~\ref{subsec:endofproof}.

\section{The ``local double counting'' approach.}
\label{sec:discussion}

Previous approaches to the problem are based on double counting arguments that, using the addability hypothesis, enable to compare the proportion of graphs having different number of connected components in a class.
The basic tool underlying the double counting, which is useful in many situations in combinatorics, is the following: to compare the sizes of two sets $\mathcal{A}$ and $\mathcal{B}$, construct a bipartite graph $H$ on $(\mathcal{A},\mathcal{B})$ in a way that we can control the degrees of vertices on each side; if all vertices from $\mathcal{A}$ have degree at least $d_\mathcal{A}$, and all vertices from $\mathcal{B}$ have degree at most $d_B$, then $|\mathcal{B}|\leq d_\mathcal{A}/d_\mathcal{B} |\mathcal{A}|$. In the context of addable classes, the roles of the sets $\mathcal{A}$ and $\mathcal{B}$ are played by $\mathcal{A}_n$ and $\mathcal{B}_n$ with previous notation, and the adjacencies in $H$ are based on the relation of edge deletion. A classical way to strengthen this method is to apply it with an \emph{edge-weighted} bipartite graph structure, for a well chosen notion of weight: this technique is useful in many situations -- for example it is this refinement that is used in  \cite{BBG, Norine}. The weights enable to make the double counting more sensitive to the particular structure of the elements of $\mathcal{A}$ and $\mathcal{B}$. 
Unfortunately, even with the use of edge-weights, approaches that are solely based on a global double counting argument seem to fall short to prove Theorem~\ref{thm:main}. Our novel approach does not consider edge-weights and exploits the local structure in a different way.

The main feature of our approach, that enables us to reach a tight bound, is that we exploit thoroughly the locality of the adjacencies in $H$. To this end, we design a ``local double counting'' strategy, based on several different steps, that we believe deserves to be described at a general level. Indeed it may be useful in many other situations where one wants to compare the size of two combinatorial sets and neither a classical double counting nor a weighted one lead to sharp enough bounds. This is the purpose of this section.
 
The locality of our approach appears at two different levels.
\emph{First}, to each element in $\mathcal{A}\cup \mathcal{B}$ we associate a statistics $\alpha$, with value in some finite dimensional space $\mathcal{E}$ that we call the parameter space, that is such that if  two elements $a \in \mathcal{A}$ and $b\in \mathcal{B}$ are linked by an edge in the bipartite graph $H$, then their corresponding $\alpha$-statistics are close in the space $\mathcal{E}$.
Moreover, if the statistics $\alpha$ is well chosen, knowing the $\alpha$-statistics of an element of $\mathcal{A}\cup \mathcal{B}$ allows us to give a more precise bound on its degree in $H$. 
For any $\alpha\in \mathcal{E}$ we can then group together, in $\mathcal{A}$ and $\mathcal{B}$, elements whose statistics belongs to a small neighbourhood of $\alpha$ into sets $\mathcal{A}_{[\alpha]}$ and $\mathcal{B}_{[\alpha]}$, and apply the double argument locally to obtain a bound on the ``local ratio'' $|\mathcal{B}_{[\alpha]}|/|\mathcal{A}_{[\alpha]}|$.

The \emph{second} way in which our approach is local is that the statistics that we choose is \emph{itself} local. Namely, in our case, the elements of $\mathcal{A}\cup \mathcal{B}$ are graphs, and the finite-dimensional parameter associated to a graph records the number of pendant copies of each tree under a certain size. It is thus a ``local observable'' -- similar to the observables underlying the local limit convergence of~\cite{BenjaminiSchramm}.

In order to make use of this abstract set-up, one needs to be able to work quantitatively with the bounds obtained from this approach: this is done as follows. On the one hand, for each element $\alpha$ in the parameter space we have, provided the previous steps were successful, a relation between each coordinate of $\alpha$ and the ratios $|\mathcal{B}_{[\alpha]}|/|\mathcal{A}_{[\alpha]}|$. On the other hand, by construction, the statistics $\alpha$ satisfies some simple \emph{global} constraints (in our case, its $L^1$-norm is smaller than $n$).
 Putting all these inequalities together, we obtain some global constraint on  a function of the ratio $|\mathcal{B}_{[\alpha]}|/|\mathcal{A}_{[\alpha]}|$, and the problem of finding the best possible upper bound on the ratio $|\mathcal{B}_{[\alpha]}|/|\mathcal{A}_{[\alpha]}|$ reduces to an \emph{optimization problem}: namely, how large can the ratio $|\mathcal{B}_{[\alpha]}|/|\mathcal{A}_{[\alpha]}|$ be, knowing that the global constraint holds. The maximum of these bounds then leads\footnote{The actual situation is a bit more technical. Indeed because the locality of the adjacencies in $H$ is only approximate, there is some overlap between the sets $\mathcal{A}_{[\alpha]}$ to take into account; this is easily handled with an averaging argument (Boxing lemma, Lemma~\ref{lemma:grid}).}
 to a bound on the ratio $|\mathcal{B}|/|\mathcal{A}|$.

In the case addressed in the present paper, this optimization problem takes the somewhat explicit form of optimizing a partition function (or generating function) of unrooted trees, given the constraint that the corresponding partition function of rooted trees is bounded. It is important to note that here, the generating functions we study are \emph{not} the generating functions of the objects in the graph class under study (which is \emph{any} bridge-addable class). It is a generating function of the ``local observables'' that we have chosen to consider, and that are the same for any graph class. Hence the role of generating functions in the present work is very different from the cases of exactly solvable models, such as random series-parallel or planar graphs~\cite{BGKN, GN}.

Another feature of our method in the present case is that in order to obtain sharper bounds, we also need to partition the set $\mathcal{B}$ into finitely many subsets $\mathcal{B}^U$ (where, in our case, the index $U$ is some unlabeled tree from a fixed finite family). This extra partitioning is not a necessary feature of the ``local double counting'' strategy we are describing, but it makes it more general.
For each $U$, we then consider the induced bipartite graph on $(\mathcal{A},\mathcal{B}^U)$, and in each case, we apply the previous ideas to get a bound on the local ratios $|\mathcal{B}^U_{[\alpha]}|/|\mathcal{A}_{[\alpha]}|$.  We then obtain a bound of the local ratio $|\mathcal{B}_{[\alpha]}|/|\mathcal{A}_{[\alpha]}|$ as a sum of these bounds. As before, an appropriate global constraint on the $\alpha$-statistics gives rise to an optimization problem, which is now \emph{multidimensional}: how large can the sum of these bounds be, given the global constraint. The number of variables of the problem is thus the number of indexing elements $U$.

\medskip

To sum up this discussion, the ``local double counting'' technique may be useful to bound the ratio of the size of two sets $\mathcal{A}$ and $\mathcal{B}$ in situations where: 
\begin{itemize}[itemsep=0pt, topsep=0pt, parsep=0pt, leftmargin=20pt]
\item[0.]
a classical double counting technique that constructs a bipartite graph on $(\mathcal{A},\mathcal{B})$ leads to an interesting bound, but not sharp; 
\item[1.]
there is a natural statistics, with value in some multidimensional parameter space (possibly of dimension arbitrarily large), that makes the bipartite graph structure ``local''; one expects this statistics to be itself a ``local measurement'' of the objects under study; 
\item[2.] 
one can write a local double counting bound that is more precise than the global one; and
\item[3.] there is some global constraint on the statistics that translates into a constraint on the local ratios;
\end{itemize}
then the method should apply, and one can get in principle a bound (either lower or upper) on the ratios as the solution of an optimization problem.
 In the case where 
\begin{itemize}[itemsep=0pt, topsep=0pt, parsep=0pt, leftmargin=20pt]
\item[4.] the set $\mathcal{B}$ may be split into several sets $\mathcal{B}^U$ in order to refine the local bounds, 
\end{itemize}
then the method applies as well, but the optimization problem becomes multidimensional. This will be the case in this paper.

\section{Proof of the main results admitting Proposition~\ref{prop:AvsB}}
\label{sec:admit}

In this section, we show how to deduce our main results (Theorems~\ref{thm:main} and ~\ref{thm:main2}) from Proposition~\ref{prop:AvsB}.
The proof of Proposition~\ref{prop:AvsB} itself is more complicated, and occupies the remaining sections.

First, the following result shows that the inequality~\eqref{eq:AvsB} relating $\mathcal{A}_n=\cG^{(1)}_n$ and $\mathcal{B}_n=\cG^{(2)}_n$ can be ``transferred'' to a larger number of connected components:
\begin{prop}\label{prop:ivsi1}
Assume that Proposition~\ref{prop:AvsB} is true.
Then for all $\epsilon'>0$, and for each $i_0 \geq 1$, there exists $n_0$ such that for all $i\leq i_0$ and $n\geq n_0$ one has:
$$
i|\cG^{(i+1)}_n| \leq   \left(\frac{1}{2}+\epsilon' \right) |\cG^{(i)}_n|
$$

\end{prop}
\noindent We note that a similar transfer principle was already used in \cite[Lemma 3.1]{ABMCR} in the context of alterable classes and with a different proof specific to that case. A similar argument was used without proof by Norin in his draft~\cite{Norine}. 
\begin{proof}

For every $i\geq 1$ and disjoint sets $V_1,\dots,V_i$, we write $\Pr(V_1,\dots,V_i)$ for the probability that for each $j\leq i$, the set $V_j$ induces a connected component in the graph $G_n$ chosen uniformly at random from $\cG_n$.
Beware that in the following we will use this notation in cases where $V_1, V_2. \dots, V_i$ is partition of $[n]=\{1,2,\dots,n\}$ but also in other cases where it is not.

We consider the following total order on the subsets of $[n]$; for every $V_1,V_2\subseteq V$ we have $V_1>V_2$ if $|V_1|>|V_2|$ or $|V_1|=|V_2|$ and the elements of $V_1$ are smaller in the lexicographical order, than the ones in $V_2$. We remark that if $V_1> \dots> V_i$ is a partition of $[n]$, then $|V_1|\geq\frac{n}{i}$.

Then, with $\uplus$ denoting the disjoint union, we have:
\begin{align*}
\Pr(G\in \cG^{(i+1)}_n) &= \sum_{V_1>\dots >V_{i+1}\atop
V_1 \uplus \dots\uplus V_{i+1} = [n]} \Pr(V_1,\dots,V_{i+1})\\
 &= \frac{1}{i} \sum_{j=1}^i \sum_{V_1>\dots >V_{i+1}\atop
V_1 \uplus \dots\uplus V_{i+1} = [n]} \Pr(V_1,\dots,V_{i+1})\\
 &= \frac{1}{i} \sum_{W_1>\dots >W_{i}\atop
W_1 \uplus \dots\uplus W_{i} = [n]} \sum_{W_1=W^{1}_1\uplus W^{2}_1 \atop
W_1^1 > W_2, W_1^1 > W_1^2}  \Pr(W^{1}_1, W^{2}_1,W_2,\dots,W_{i}),
\end{align*}
where the last equality is just a change of index, that consists in noting $V_1=W_1^1$, $V_{j+1}=W_1^2$, $W_1=V_1\uplus V_{j+1}$, and in summing first over the set $W_1$ and then over its partitions into two sets (here $W_2, \dots, W_i$ denote the remaining sets $V_k$ for $k\not\in \{ 1,j+1\}$). Note the constraint $W^1_1>W_2$, that comes from the fact that $V_1$ is the largest set of the partition (for $<$)  before the change of index. If we remove this constraint, we only make the sum larger and we get the upper bound:
\begin{align*}
\Pr(G\in \cG^{(i+1)}_n)
  &\leq \frac{1}{i} \sum_{W_1>\dots >W_{i}\atop
W_1 \uplus \dots\uplus W_{i} = [n]} \sum_{W_1=W^{1}_1\uplus W^{2}_1 \atop W_1^1 > W_1^2
}  \Pr(W^{1}_1, W^{2}_1,W_2,\dots,W_{i}) \\
\end{align*}
For any set $W\subseteq V$ we write $G[W]$ for the graph induced by the set of vertices in $W$. For any graph $Z$, we denote by $G[W]\cong Z$ the event that $W$ induces a connected component that it is isomorphic to $Z$.
Given a partition $W_1 \uplus \dots \uplus W_i = V$ such that $W_1>\dots>W_i$ and graphs $Z_2,\dots Z_i$,
consider the graph class
$$
\mathcal{H} (W_1,W_2\dots,W_i;Z_2,\dots, Z_i)
=\{G[W_1]:\;G\in \cG_n \text{ and } G[W_j]=Z_j \text{ for }2\leq j\leq i\}\;.
$$
Notice that $\mathcal{H}(W_1,W_2\dots,W_i;Z_2,\dots, Z_i)$ is a bridge-addable class on the set of vertices $W_1$. Let $H$ be a graph chosen uniformly at random from $\mathcal{H}(W_1,W_2\dots,W_i;Z_2,\dots, Z_i)$. By the remark above, $H$ has order at least $\frac{n}{i+1}$. Let $n_0$ such that $\frac{n_0}{i_0+1}\geq n_0(\epsilon')$, where $n_0(\epsilon')$ is the constant that appears in Proposition~\ref{prop:AvsB}. Using this proposition we have
\begin{align*}
\sum_{W_1=W^{1}_1\cup W^{2}_1\atop W_1^1 > W_1^2}
\Pr(W^{1}_1, W^{2}_1,W_2,\dots,W_{i}) 
 &= \sum_{Z_2,\dots,Z_i}\Pr(H\in \mathcal{H}^{(2)} ( W_1,W_2\dots,W_i;Z_2,\dots, Z_i))\\
&\hspace{1cm}\cdot\Pr(H\in  \mathcal{H}(W_1,W_2\dots,W_i;Z_2,\dots, Z_i)\\
&\leq \sum_{Z_2,\dots,Z_i} \left(\frac{1}{2}+\epsilon' \right) \Pr(H\in \mathcal{H}^{(1)}( W_1,W_2\dots,W_i;Z_2,\dots, Z_i))\\
&\hspace{1cm} \cdot\Pr(H\in \mathcal H( W_1,W_2\dots,W_i;Z_2,\dots, Z_i))\\
&= \left(\frac{1}{2}+\epsilon' \right) \Pr(W_1, W_2,\dots,W_{i})\;.
\end{align*}
Thus returning to the previous bound we obtain:
\begin{align*}
i\Pr(G\in \cG^{(i+1)}_n) 
&\leq  \left(\frac{1}{2}+\epsilon' \right) \sum_{W_1>\dots >W_{i}\atop
W_1 \uplus \dots\uplus W_{i} = [n]} \Pr(W_1,W_2,\dots,W_{i})\\
&=  \left(\frac{1}{2}+\epsilon' \right) \Pr(G\in \cG^{(i)}_n)\;. \qedhere
\end{align*}

\end{proof}

To conclude the proof of Theorem~\ref{thm:main}, we will also need the following observation from~\cite{BBG}. We include the proof for the sake of completeness.
\begin{lemma}[\cite{BBG}]\label{lemma:simpleCounting}
For each $i, n\geq 1$ one has:
$$
i|\cG^{(i+1)}_n| \leq  |\cG^{(i)}_n|.
$$
\end{lemma}
\begin{proof}
Construct a bipartite graph $H$ on the vertex set $(\cG^{(i)}_n, \cG^{(i+1)}_n)$ by adding an edge between $G_1 \in \cG^{(i)}_n$ and $G_2 \in \cG^{(i+1)}_n$ if $G_2$ can be obtained from $G_1$ by removing an edge. Note that a graph $G_1\in \cG^{(i)}_n$ has degree at most $n-i$ in $H$, since $G_1$ has at most $n-i$ cut-edges. Moreover a graph $G_2 \in \cG^{(i+1)}_n$ has degree at least $i(n-i)$, by the property of bridge-addability. By counting the edges of $H$ in two different ways, we thus get:
$$
i(n-i)|\cG^{(i+1)}_n| \leq |E(H)| \leq  (n-i) |\cG^{(i)}_n|,
$$
which concludes the proof.
\end{proof}
\medskip

We first prove the main theorem.
\begin{proof}[Proof of Theorem~\ref{thm:main}]

Let $\epsilon'=\ln((1-\epsilon)^{-1}-\epsilon)$ and let $x=\frac{1}{2}+\epsilon'$. Set $i_0$ to be large enough such that $\frac{1}{i_0!}<\epsilon\cdot e^{1/2}$.
Using Proposition~\ref{prop:ivsi1} recursively with $\epsilon'$ and $i_0$, we have that there exists an $n_0$ such that for every $i\leq i_0$ and for every $n\geq n_0$
$$
|\cG^{(i+1)}_n|\leq \frac{x^i}{i!}|\mathcal{A}_n|\;.
$$
Moreover, for every $i>i_0$ we have from Lemma~\ref{lemma:simpleCounting}:
$$
|\cG^{(i+1)}_n|\leq \frac{1}{i!}|\mathcal{A}_n|\;.
$$
Using both inequalities we obtain that for any $n\geq n_0$
\begin{align*}
  |\cG_n| &=\sum_{i=1}^n |\cG^{(i)}_n| \leq \left(\sum_{i=1}^{i_0}\frac{x^{i-1}}{(i-1)!} + \sum_{i=i_0+1}^{n}\frac{1}{i!}  \right) |\mathcal{A}_n| \leq \left( e^x+ \frac{1}{i_0!} \right) |\mathcal{A}_n|\\
  &\leq \left( e^{\frac{1}{2}+\epsilon'}+ \epsilon \cdot e^{1/2} \right) |\mathcal{A}_n|= (1-\epsilon)^{-1}  e^{1/2} |\mathcal{A}_n|\;.
\end{align*}
The probability that a graph $G_n$ chosen uniformly at random from $\cG_n$ is connected is
$$
\Pr \left( G_n \mbox{ is connected} \right)
=
\frac{|\mathcal{A}_n|}{|\cG_n|}\geq (1-\epsilon) e^{-1/2}\;,
$$
provided that $n\geq n_0$.
\end{proof}

We conclude by proving the extension of our main result.
\begin{proof}[Proof of Theorem~\ref{thm:main2}]

Let $n$ be a large constant (to be fixed later) and let $M_n$ be the number of components of a graph chosen uniformly at random in $\cG_n$ minus one. By Lemma~\ref{lemma:simpleCounting}, there exists $i_0\geq 1$ such that 
\begin{align}\label{eq:tail}
\Pr(M_n\geq i_0+1)< \epsilon/2\;.
\end{align}
We can assume that $k\leq i_0$ since otherwise we are done by~\eqref{eq:tail}.

We will use a result on stochastic domination of Poisson distributions given by McDiarmid~(see Lemma 3.3 in~\cite{Mrandom2012}). Given $\alpha>0$ and $k_0\geq 1$, if $X$ is a non-negative integer-valued random variable such that for every $k=0,1,\dots, k_0-1$ we have
$$
\Pr(X= k+1)\leq \frac{\alpha}{k+1} \Pr(X=k)\;,
$$
then 
\begin{align}\label{eq:tail2}
\Pr(k_0\geq X\geq k)\leq \Pr( \mbox{Poisson}(\alpha)\geq k)\;.
\end{align}

By Proposition~\ref{prop:ivsi1} with the chosen $i_0$, for every $\epsilon'$ there exists an $n_0$ such that if $n\geq n_0$, the random variable $M_n$ satisfies the hypothesis of Lemma 3.3 in~\cite{Mrandom2012} with $\alpha=\frac{1}{2}+\epsilon'$. Fix $k_0=i_0$. Using~\eqref{eq:tail} and~\eqref{eq:tail2}, we obtain
\begin{align*}
\Pr( G_n \mbox{ has at most $k\!+\!1$ connected components})
&=\Pr(M_n\leq k)\\
&=1-\Pr(i_0\geq M_n\geq k+1)- \Pr(M_n\geq i_0+1)\\ 
&\geq 1- \Pr( \mbox{Poisson}(1/2+\epsilon')\geq k+1) -\epsilon/2\\ 
&= \Pr( \mbox{Poisson}(1/2+\epsilon')\leq k) -\epsilon/2\\
&= \Pr( \mbox{Poisson}(1/2)\leq k) -\epsilon\;,
\end{align*}
where the last inequality follows provided that $\epsilon'$ is small enough with respect to $\epsilon$.
\end{proof}

\section{Local double counting and local parameters}\label{sec:local}

We now start the proof of the main technical estimate, namely Proposition~\ref{prop:AvsB}. The following reduction will be very useful:
\begin{lemma}[{\cite[Lemma~2.1]{BBG}}]
Assume that Proposition~\ref{prop:AvsB} is true under the additional assumption that all graphs in $\mathcal{A}_n$ and $\mathcal{B}_n$ are forests. Then it is also true without this assumption.
\end{lemma}
\noindent The proof of this lemma relies on a simple and beautiful argument that consists in splitting the sets $\mathcal{A}_n$ and $\mathcal{B}_n$ into equivalence classes depending on their $2$-edge connected blocks, and then choosing a spanning tree arbitrarily in each block. This construction transforms any bridge-addable class of graphs into several bridge-addable classes of forests while preserving the distribution of the number of components, from which the lemma easily follows. We refer to \cite{BBG} for the full proof.

\medskip
Thanks to the last lemma, for the rest of the paper, we will make the following assumption:
\smallskip

\noindent{\bf Assumption:}
{\it For all $n\geq 1$, all graphs in $\mathcal{A}_n$ and $\mathcal{B}_n$ are forests. }

\subsection{Local parameters and partitions}\label{subsec:localparams}

In order to compare the sizes of $\mathcal{A}_n$ and $\mathcal{B}_n$, we will refine the double counting technique used in the proof of Lemma~\ref{lemma:simpleCounting}. 
We will again construct a bipartite graph on the vertex set $(\mathcal{A}_n, \mathcal{B}_n)$ where an edge is placed between $G_1\in \mathcal{A}_n$ and $G_2\in \mathcal{B}_n$ if one can be obtained from the other by the deletion of an edge.
However, in order to obtain more precise bounds on the degrees of vertices in this bipartite graph, we will partition the sets $\mathcal{A}_n$ and $\mathcal{B}_n$ according to some local parameters of the graphs. Namely, to each graph $G$ we will associate a statistics $\alpha^G$ that records, roughly speaking, the number of pendant copies in $G$ of each tree from some finite family $\mathcal{T}_0$. The vectors $\alpha^G$ will be elements of a space called the parameter space and denoted by $\mathcal{E}$.  The set $\mathcal{B}_n$ will be further partitioned according to the isomorphism type of the smallest component, and a special role will be played by subsets where this smallest component belongs to a finite family of trees called $\mathcal{U}_0$. The purpose of this subsection is to set notation and define these partionings.

We write $\cT$ for the family of all rooted unlabeled trees and $\cU$ for the family of all unrooted unlabeled trees. We also use $\cT^{\ell}$ and $\cU^{\ell}$ for the corresponding sets of \emph{labeled} objects.
For every tree $U\in\cU$ we note by $\Aut_u(U)$ for the total number of automorphisms of $U$.
For every $T\in \cT$ we note by $\Aut_r(T)$ the number of automorphisms of $T$ as a rooted tree (i.e. the number of automorphisms that fix the root of $T$).

In this section we will fix two finite families $\cU_0\subset \cU$ and $\cT_0\subset \cT$ such that
\begin{itemize}[itemsep=2pt, topsep=0pt, parsep=0pt, leftmargin=30pt]
\item $\cU_0$ contains the only unrooted tree of order one (a single vertex).
\item $\cT_0$ is closed under rooted inclusion; that is, if $T\in \cT$ with root at $v$ and $T'\subseteq T$ is a subtree that contains the root, then $T'\in \cT$.
\end{itemize}
We write $t_{\max}$ and $u_{\max}$ for the maximum number of vertices of a tree in $\cT_0$ and $\cU_0$ respectively.

Following the definition in~\cite{ABMCR}, for each tree $G$ and each edge $e$, the \emph{pendant tree} of $G$ in $e$ is denoted by $s(G,e)$ and it is the smallest component of $G- e$, or the component containing vertex $1$ if the components have equal size. Since the tree $G$ has $n-1$ edges, every tree on $n$ vertices has exactly $n-1$ pendant trees.

Define the \emph{parameter space} $\cE=[0..(n-1)]^{\cT_0}$. To each tree $G \in \mathcal{A}_n$ we associate the vector $\alpha^G\in \cE$ such that, $\alpha^G(T)$ is the number of pendant copies of $T$ in $G$. Precisely, for every rooted tree $T\in \mathcal{T}_0$, $$\alpha^G(T):=|\{e\in E(G):\; s(G,e)\cong T\}|,$$
where the symbol $\cong$ denotes isomorphisms of \emph{rooted} trees.
Then, for each tree $G\in \mathcal{A}_n$, we have that
\begin{eqnarray}\label{eq:constraintSumAlphaDiscrete}
\sum_{T \in \cT_0} \alpha^G(T) \leq n-1.
\end{eqnarray}
The definition of $\alpha^G(T)$ can be extended to forests as follows. If $G$ is a forest with connected components $G_1\dots G_i$, where $G_1$ is the largest component (or, if there is ambiguity, the one containing the smallest vertex among the largest ones), then $\alpha^G(T)=\alpha^{G_1}(T)$ for every rooted tree $T\in \mathcal{T}_0$.

For every $\alpha\in \cE$, we use the notation $\mathcal{A}_{n,\alpha}$ to denote the set of graphs $G\in \mathcal{A}_n$ such that $\alpha^G=\alpha$. In other words, $\mathcal{A}_{n,\alpha}$ is the set of graphs in $\mathcal{A}_n$ that have precisely $\alpha(T)$ pendant trees that are isomorphic to $T$, for each $T\in \cT_0$. This partitions $\mathcal{A}_n$ into different sets according to their $\alpha$-statistics:
$$
\mathcal{A}_n= \biguplus_{\alpha \in \cE} \mathcal{A}_{n,\alpha}\;.
$$
Note that by \eqref{eq:constraintSumAlphaDiscrete} one can restrict the previous union to vectors $\alpha$ satisfying $\sum_{T\in\cT_0} \alpha(T) \leq n-1$. For every $\Gamma\subseteq \cE$, we also define
$$
\mathcal{A}_{n,\Gamma} = \biguplus_{\alpha \in \Gamma} \mathcal{A}_{n,\alpha}\;.
$$

For every unrooted tree $U \in \cU$, we let $\mathcal{B}_n^U$ be the subset of $\mathcal{B}_n$ composed by graphs whose smallest component (or the one containing vertex $1$ if they have the same size) is isomorphic to~$U$. All the previous definitions for $\mathcal{A}_n$ extend naturally to the sets $\mathcal{B}_n$ and $\mathcal{B}^U_n$. That is, we can partition the set $\mathcal{B}_n$ according to $U\in \cU$ and the $\alpha$-statistics:
$$
\mathcal{B}_n= \biguplus_{U \in \cU} \mathcal{B}_n^U= \biguplus_{U \in \cU} \biguplus_{\alpha \in \cE} \mathcal{B}^U_{n,\alpha}\;.
$$

\subsection{The main local double counting lemma}
\label{subsec:localcounting}

In this subsection we construct the promised bipartite graph structure on $(\mathcal{A}_n,\mathcal{B}_n)$, and we analyse \emph{locally} the degrees of this graph. This enables us, for each $\alpha$ in the parameter space $\mathcal{E}$, to compare the number  of graphs $G$ in $\mathcal{A}_n$ and $\mathcal{B}_n$ whose statistics $\alpha^G$ is close to $\alpha$ (Corollary~\ref{cor:boundforTe}).

Consider a rooted tree  $T \in \cT$ and an edge $e$ of $T$. If we remove $e$ from $T$, we obtain two connected components: we note $T_-\in \cT$ the one containing the root of $T$, and $U_+\in \cU$ the other one. We let $v_-\in V(T_-)$ and $v_+\in V(U_+)$ be the two endpoints of $e$. We emphasize that $T_-$ is considered as a rooted tree (rooted at the root of $T$), but $U_+$ is considered as an unrooted tree.

\begin{definition}[Multiplicities of edges and vertices in rooted or unrooted trees, see Figure~\ref{fig:multiplicities}]\label{def:multiplicities}
Let $T, e, T_-, U_+, v_-, v_+$ be as above. Then:
\begin{itemize}[itemsep=0pt, topsep=0pt, parsep=0pt, leftmargin=10pt]
\item[-]
The \emph{rooted multiplicity of the edge $e$ in $T$}, denoted by $m_T(e)$ is the number of distinct edges in $T$ that are mapped to $e$ by some isomorphism of $T$ (as a rooted tree).
\item[-]
The \emph{rooted multiplicity of the vertex $v_-$ in $T_-$}, denoted by $m_{T_-}(v_-)$ is the number of distinct vertices in $T_-$ that are mapped to $v_-$ by some isomorphism of $T_-$ (as a rooted tree, rooted at the root of $T$).
\item[-]
The \emph{unrooted multiplicity of $v_+$ in $U_+$}, denoted by $n_{U_+}(v_+)$ is the number of distinct vertices in $U_+$ that are mapped to $v_+$ by some isomorphism of $U_+$ (as an unrooted tree).
\end{itemize}
\end{definition}
\begin{figure}[h!!]
\begin{center}
\includegraphics[width=0.8\linewidth]{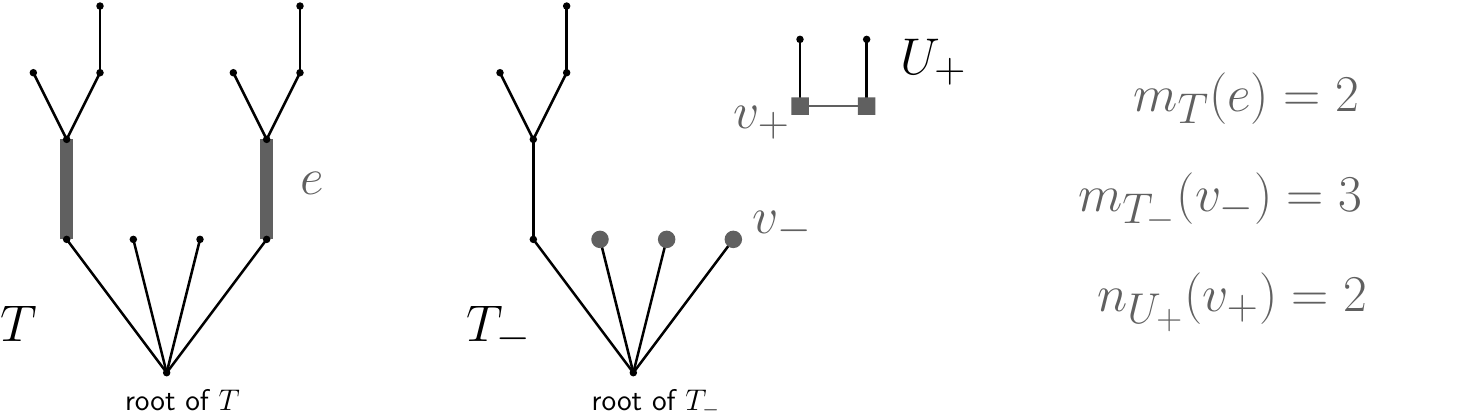}
\caption{Construction of $T_-$ and $U_+$ from $T$ and $e$, and the relevant multiplicities.}
\label{fig:multiplicities}.
\end{center}
\end{figure}

In order to state our main combinatorial lemma, we first need to introduce the concept of boxes.
For every $\alpha\in \cE$ and $w\geq 1$ we define the \emph{box} $[\alpha]^w\subset \cE$ as the parallelepiped:
$$
[\alpha]^w:=\{\alpha'\in \cE: \ \forall T \in \cT_0, \ \alpha(T) \leq \alpha'(T) < \alpha(T)+w\}.
$$
The parameters $\alpha$ and $w$ of the box $[\alpha]^w$ will be referred to as its \emph{lower corner} and its \emph{width}. We also define the $q$-neighbourhood of $[\alpha]^w$ as the set of elements in $\cE$ that are at distance at most $q$ from the box. Precisely,
$$
[\alpha]^w_q:=\{\alpha'\in \cE: \ \forall T \in \cT_0, \ \alpha(T)-q \leq \alpha'(T) < \alpha(T)+w+q\}.
$$
Note that in many cases $[\alpha]^w_q$ is itself a box, but for the structure of our argument it will be convenient to think of it as a neighbourhood of the box $[\alpha]^w$ in $\mathcal{E}$, hence our different  notation.


Here we show the crucial double counting argument that will allow us to compare the sets $\mathcal{A}_n$ and $\mathcal{B}_n$.
\begin{lemma}[Local double counting lemma]\label{lemma:localSwitching}
There exists a constant $q_*=q_*(\cU_0)\leq u_{\max}$ (and independent from $n$) such that the following is true.
Let $T \in \cT_0$ and let $e$ be an edge of $T$. Let $T_-, U_+, v_-, v_+$ be as above and assume that $T_- \in \cT_0$ and $U_+ \in \cU_0$.

Then for every $\alpha\in \cE$ and $w\geq 1$ one has:
\begin{eqnarray}\label{eq:localSwitchingEq}
m_{T}(e)\cdot  \big(\alpha(T)+w+q_*\big) \cdot \left| \mathcal{A}_{n,[\alpha]^w_{q_*}}\right|
\geq n_{U_+}(v_+)\cdot  m_{T_-}(v_-) \cdot  \alpha(T_-) \cdot \left|\mathcal{B}^{U_+}_{n,[\alpha]^w}\right|\;.
\end{eqnarray}
\end{lemma}
\begin{proof}

The proof of this lemma consists on a simple double counting on the edges of the bipartite graph $H$ with vertex sets $(\mathcal{A}_{n,[\alpha]^w_{q_*}},\mathcal{B}^{U_+}_{n,[\alpha]^w})$ and where $G_1\in \mathcal{A}_{n,[\alpha]^w_{q_*}}$ and $G_2\in \mathcal{B}^{U_+}_{n,[\alpha]^w}$ are joined by an edge if and only if there exists a copy of $T$ in $G_1$ and an edge $e$ in $T$ such that
$G_1-e= G_2$,
$T - e= U_+\cup T_-$ and $s(G_1,e)=U_+$.

On the one hand, the number of edges in $H$ is $|E(H)|= \sum_{G_2\in \mathcal{B}^{U_+}_{n,[\alpha]^w}} \deg_H(G_2)$, where $\deg_H(G)$ denotes the degree of $G$ in $H$. Since $\cG_n$ is a bridge-addable class of graphs, we can add any bridge to $G_2$ and stay in the class. Moreover, if $G_1$ can be obtained from $G_2$ by adding a bridge, then $G_1$ belongs to the $q_*=|U_+|\leq u_{\max}$ neighbourhood of $\alpha$; that is, gluing a tree of size $|U_+|$ can change the statistics of $G_2$ by at most $|U_+|$ in each component (for every $T$, $|\alpha^{G_2}(T)-\alpha^{G_1}(T)|\leq |U_+|$). Then, an easy argument shows that $\deg_H(G_2)=n_{U_+}(v_+) m_{T_-}(v_-)\alpha^{G_2}(T_-)$. Thus,
$$
|E(H)|= \sum_{G_2\in \mathcal{B}^{U_+}_{n,[\alpha]^w}} n_{U_+}(v_+) m_{T_-}(v_-)\alpha^{G_2}(T_-) \geq  n_{U_+}(v_+) m_{T_-}(v_-) \alpha(T_-) \left|\mathcal{B}^{U_+}_{n,[\alpha]^w}\right|\;.
$$

On the other hand, the number of edges in $H$ is $|E(H)|= \sum_{G_1\in \mathcal{A}_{n,[\alpha]^w_{q_*}}} \deg_H(G_1)$. In this case, it suffices to upper bound the degree of $G_1$ into the set $\mathcal{B}^{U_+}_{n,[\alpha]^w}$. A simple argument also shows that
$\deg_H(G_1)\leq m_{T}(e) \alpha^{G_1}(T)$. Thus,
$$
|E(H)|\leq  \sum_{G_1\in \mathcal{A}_{n,[\alpha]^w_{q_*}}}  m_{T}(e) \alpha^{G_1}(T) \leq  m_{T}(e) (\alpha(T)+w+q_*) \left|\mathcal{A}_{n,[\alpha]^w_{q_*}}\right|\;.
$$
The lemma follows from the two previous inequalities.
\end{proof}

\begin{rem}\label{rem:1}
The last lemma is also valid in the following degenerate case. Assume that $T$, viewed as an unrooted tree, is an element of $\mathcal{U}_0$, and let $U_+=T$ (considered as an unrooted tree). Let conventionally
$T_-:=\emptyset$. Then~\eqref{eq:localSwitchingEq} holds with the conventions $\Aut_r(\emptyset)=1$, $\alpha(\emptyset)=n-|T|$, $m_T(e)=1$. In other words, we have, with $v_+$ the root of $T$:
\begin{eqnarray*}
 \big(\alpha(T)+w+q_*\big) \cdot \left| \mathcal{A}_{n,[\alpha]^w_{q_*}}\right|
\geq n_{T}(v_+) (n-|T|) \cdot \left|\mathcal{B}^{T}_{n,[\alpha]^w}\right|\;.
\end{eqnarray*}
The proof is similar: we just consider the bipartite graph structure on $(\mathcal{A}_{n,[\alpha]^w_{q_*}},\mathcal{B}^{U_+}_{n,[\alpha]^w})$ defined by the fact that $G_1\in \mathcal{A}_{n,[\alpha]^w_{q_*}}$ and $G_2\in \mathcal{B}^{T}_{n,[\alpha]^w}$ are joined by an edge if and only if there exists a copy of $T$ in $G_1$ and an edge $e$ in $G_1$ such that $s(G_1,e)=T$. The only thing to note is that for each  $G_2\in \mathcal{B}^{T}_{n,[\alpha]^w}$, there are $n_{T}(v_+) (n-|T|)$ ways to add a bridge to $G_2$ between one of the $n_{T}(v_+)$ allowed vertices of its connected component isomorphic to $T$, to any of the $(n-|T|)$ vertices of its other connected component.
\end{rem}

\begin{lemma}\label{lemma:auto}
With the above notation, we have:
$$
\frac{m_T(e)}{\Aut_r(T)} = \frac{m_{T_-}(v_-) n_{U_+}(v_+)}{\Aut_r(T_-)\Aut_u(U_+)}
$$
\end{lemma}
\begin{proof}
  We will prove the equality by counting the \emph{labeled} rooted trees in $\cT^\ell$ with a marked edge $e$ that are isomorphic to the unlabeled rooted tree $T\in \cT$. Recall that $T$ and $e$ are such that $T-e=T_-\cup U_+$. On the one hand, there are $|T|!/(\Aut_r(T_-)\Aut_u(U_+))$ different ways to label $T_-\cup U_+$. By definition of $m_{T_-}(v_-)$ and $n_{U_+}(v_+)$, for each of these labellings, there are $m_{T_-}(v_-) n_{U_+}(v_+)$ ways to select an edge $e$ to connect $T_-$ and $U_+$, such that we obtain a labelling of the tree $T$ with marked edge $e$ (note that all these choices are inequivalent since we work with a labeled structure). On the other hand, there are $|T|!/\Aut_r(T)$ different labellings of the tree $T$, and by definition each of them gives raise to $m_T(e)$ ways to mark the edge $e$, by the previous construction. Hence $ m_T(e)\cdot |T|!/\Aut_r(T) =m_{T_-}(v_-) n_{U_+}(v_+) \cdot |T|!/(\Aut_r(T_-)\Aut_u(U_+))$, and the lemma is proved.
\end{proof}

The next corollary follows immediately:
\begin{cor}\label{cor:boundforTe}
Let $U\in \mathcal{U}_0$.
Then for any $T,e,T_-,U_+$ as in Definition~\ref{def:multiplicities} such that $U_+\cong U$, and for any $\alpha\in \cE$ and $w\geq 1$,
 one has:
$$
\Aut_u(U)\cdot 
\left|\mathcal{B}_{n,[\alpha]^w}^U\right|
 \leq
\frac{\Aut_r(T)}{\Aut_r(T_-)} \cdot\frac{\alpha(T)+w+q_*}{\alpha(T_-)}
\cdot
\left|\mathcal{A}_{n,[\alpha]^w_{q_*}}\right|,
$$
where $q_*=q_*(\cU_0)$ is the constant obtained in Lemma~\ref{lemma:localSwitching}.
\end{cor}
\noindent Observe that the last corollary also holds in degenerate case $U_+=T$, $T_-=\emptyset$, with the notation of Remark~\ref{rem:1}.

\subsection{Inductive bounds and tree weights}\label{sec:weights}

In this subsection we iterate the bound of Corollary~\ref{cor:boundforTe} to obtain, for each $\alpha \in \mathcal{E}$ and $T\in \mathcal{T}_0$, a lower bound on $\alpha(T)$ in terms of the ratios $\left|\mathcal{B}_{n,[\alpha]^w}^U\right| / \left|\mathcal{A}_{n,[\alpha]^w_{q_*}}\right|$ for different $U\in \mathcal{U}_0$. We then use inequality~\eqref{eq:constraintSumAlphaDiscrete} to conclude that a certain functional of these ratios is bounded (Corollary~\ref{cor:sumBound}). This is the main combinatorial step towards proving a bound on the sum of these ratios, which is the quantity revelant to prove Proposition~\ref{prop:AvsB} (this will be done in the next section).

Let $T\in \cT$ be a rooted tree. A \emph{$\cU_0$-admissible decomposition  of $T$} is an increasing sequence $\mathbf{T}=(T_i)_{i\leq \ell}$ of labeled trees $$T_1\subset \dots \subset T_\ell =T$$ for some $\ell \geq 1$ called the \emph{length}, such that $T_1\in \mathcal{U}_0$ and that for each $2\leq i \leq\ell$, $T_{i}$ is obtained by joining $T_{i-1}$ by an edge $e_i$ to some tree $U_i\in \cU_0$.

By the choices made at the beginning of Section~\ref{subsec:localparams} for $\cU_0$ (that contains the tree of size one) and $\cT_0$ (that it is closed by inclusion), if $T\in \cT_0$, then $T$ has at least one $\cU_0$-admissible decomposition such that $T_i\in \cT_0$ for every $1\leq i \leq\ell$.

Fix $\alpha\in \cE$, $w\geq 1$ and let $q_*=q_*(\cU_0)$ be the constant obtained from Lemma~\ref{lemma:localSwitching}. Throughout this subsection we will focus on the box $[\alpha]^w$. Let $\Lambda=(\mathbb{R}_+)^{\cU_0}$. We define the vector $\mathbf{z}_{n,\alpha}=\mathbf{z}_{n,[\alpha]^w_{q_*}} = (z_{n,\alpha}^U)_{U \in \cU_0}\in \Lambda$ by:
\begin{align}\label{eq:defzn}
z_{n,\alpha}^U:=\Aut_u(U) \frac{|\mathcal{B}^U_{n,[\alpha]^w}|}{|\mathcal{A}_{n,[\alpha]^w_{q_*}}|} \left(1-\frac{|U|}{n}\right).
\end{align}
Observe that if $|\mathcal{B}^U_{n,[\alpha]^w}|>0$, then, since the class of graphs is bridge-addable, we have $|\mathcal{A}_{n,[\alpha]^w_{q_*}}|>0$. If $|\mathcal{B}^U_{n,[\alpha]^w}|=0$, then we set $z^U_{n,[\alpha]^w}:=0$.

For any $\mathbf{z}\in \Lambda$, any $T\in \cT$ and any $\cU_0$-admissible decomposition of it $\mathbf{T}=(T_i)_{i\leq \ell}$, the \emph{weight} of $\mathbf{T}$ with respect to $\mathbf{z}$ is defined as $\omega (\mathbf{T}, \mathbf{z}) =\prod_{i=1}^{\ell} z^{U_i}$, where $U_i=T_i\setminus T_{i-1}$ as an unrooted tree~(here we use the convention $T_0=\emptyset$).
\begin{figure}[h!!]
\begin{center}
\includegraphics[width=0.6\linewidth]{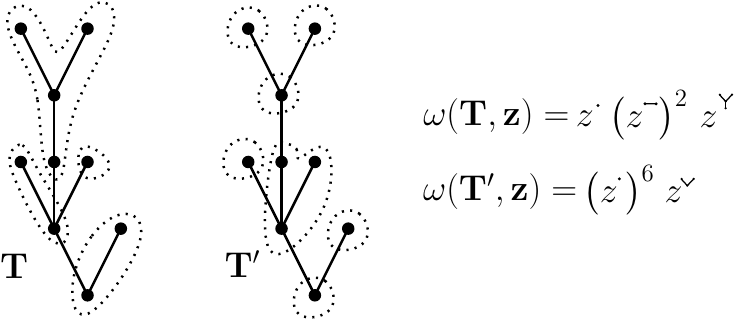}
\caption{Two $\mathcal{U}_0$-admissible decompositions $\mathbf{T}$ and $\mathbf{T}'$ of the same tree, and the corresponding weights. In this case we assume that the trees
 \protect\raisebox{-1mm}{\protect\includegraphics[scale=1.4]{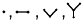}}
 belong to $\mathcal{U}_0$.}
\label{fig:decomposition}.
\end{center}
\end{figure}

\begin{lemma}\label{lemma:Bound}
For any $T \in \cT$ and any $\cU_0$-admissible decomposition $\mathbf{T}$ of $T$ of length $\ell$, one has:
$$
\frac{\alpha(T)}{n} \geq \frac{\omega(\mathbf{T},\mathbf{z}_{n,\alpha})}{\Aut_r(T)} -\frac{(w+q_*) (2|T|)^{\ell-1}}{n}\;.
$$
\end{lemma}
\begin{proof}
Let $ T_1\subset \dots \subset T_\ell =T$ be the $\cU_0$-admissible decomposition $\mathbf{T}$.
We will show the statement using induction on the length of $\mathbf{T}$.

If $\ell=1$, then $T$ is a rooted copy of $U$ for some unrooted tree $U\in \cU_0$. By Remark~\ref{rem:1} we have,
\begin{align*}
\frac{\alpha(T)}{n}&\geq  \frac{\Aut_u(U)\Aut_r(\emptyset)\alpha(\emptyset)\left|\mathcal{B}_{n,[\alpha]^w}^U\right|}{\Aut_r(T)n\left|\mathcal{A}_{n,[\alpha]^w_{q_*}}\right|} -\frac{w+q_*}{n}
=\frac{z^U_{n,\alpha}}{\Aut_r(T)} -\frac{w+q_*}{n}
=\frac{\omega(\mathbf{T},\mathbf{z}_{n,\alpha})}{\Aut_r(T)} -\frac{w+q_*}{n}\;,
\end{align*}
where we used the conventions $\Aut_r(\emptyset)=1$ and $\alpha(\emptyset)=n-|U|$ and the definition of $z^U_{n,\alpha}$.

Let us assume that the inequality is true for every tree $T_-\in \cT$ and for every $\cU_0$-admissible decomposition of $T_-$ of size at most $\ell-1$. Let $\mathbf{T}_-$ be the admissible decomposition induced by $\mathbf{T}$ in $T_-=T_{\ell-1}$. Then letting $U=T \setminus T_-$, we have by Lemma~\ref{lemma:auto}:
$$
\Aut_r(T_-)=  \frac{m_{T_-}(v_-)n_{U}(v_+)}{m_T(e)\Aut_u(U)}\cdot \Aut_r(T)\leq |T|\Aut_r(T)\;,
$$
since $n_{U}(v_+)\leq \Aut_u(U)$ and $m_{T_-}(v_-)\leq |T_-|\leq |T|$.

By using Corollary~\ref{cor:boundforTe} and the induction hypothesis on $T_-$ we obtain
\begin{align*}
\frac{\alpha(T)}{n}&\geq  \frac{\Aut_u(U)\Aut_r(T_-)|\mathcal{B}_{n,[\alpha]^w}^U|}{\Aut_r(T)n|\mathcal{A}_{n,[\alpha]^w_{q_*}}|}\alpha(T_-) -\frac{w+q_*}{n} \\
&=\frac{z^U_{n,\alpha} \Aut_r(T_-)}{\Aut_r(T)}\cdot \frac{\alpha(T_-)}{n-|U|} -\frac{w+q_*}{n} \\
&\geq \frac{z^U_{n,\alpha} \omega(\mathbf{T}_-,\mathbf{z}_{n,\alpha})}{\Aut_r(T)} -\frac{w+q_*}{n}\left(1+\frac{\Aut_r(T_-)(2|T|)^{\ell-2}}{\Aut_r(T)}\right) \\
& \geq \frac{\omega(\mathbf{T},\mathbf{z}_{n,\alpha})}{\Aut_r(T)} -\frac{(w+q_*) (2|T|)^{\ell-1}}{n}\;.
\end{align*}
\end{proof}

\begin{definition}
For any $\mathbf{z}\in \Lambda$ and any $T\in\cT$, we define its \emph{maximum weight} with respect to $\mathbf{z}$, denoted by $\omega(T,\mathbf{z})$, to be the largest weight $\omega (\mathbf{T}, \mathbf{z})$, where $\mathbf{T}$ is a $\cU_0$-admissible decomposition of~$T$.
Note that $\omega(T,\mathbf{z})$ is well defined since each tree $T$ has at least one $\cU_0$-admissible decomposition\footnote{Here we do a slight abuse of notation by using $\omega(\mathbf{T},\mathbf{z})$ and $\omega(T,\mathbf{z})$ to denote, respectively, the weight of a given decomposition $\mathbf{T}$ and the maximum weight of a decomposition of a given tree $T$.}.
\end{definition}

We introduce the following weighted sum, for $\mathbf{z}\in \Lambda$:
\begin{align*}
Y_{\mathcal{T}_0}(\mathbf{z}):=\sum_{T\in\cT_0} \frac{\omega(T, \mathbf{z})}{\Aut_r(T)}.
\end{align*}
Then we immediately have from the previous lemma:
\begin{cor}\label{cor:sumBound}
Assume that $\alpha\in \cE$ is such that $\sum_{T\in\cT_0} \alpha(T) \leq n-1$. Then one has
\begin{eqnarray}\label{eq:sumBound}
Y_{\mathcal{T}_0}(\mathbf{z}_{n,\alpha}) \leq 1+ \frac{C}{n},
\end{eqnarray}
where $C=(w+q_*)(2t_{max})^{t_{max}-1} |\cT_0|$ is a constant depending only on $\cT_0$, $\cU_0$ and $w$ (but not on $n$).
\end{cor}
\begin{proof}
This is proved by summing the upper bound of Lemma~\ref{lemma:Bound} for the $\cU_0$-admissible decomposition that gives the maximum weight and over all $T\in \cT_0$.
\end{proof}

\section{Partition functions and optimization}\label{sec:partitionfunctions}

In Section~\ref{sec:local} we have obtained, for each $\alpha \in \mathcal{E}$, a bound  on a functional of the ratios $z^U_{n,\alpha}$ for $U\in\cU_0$ (Corollary~\ref{cor:sumBound}). Note that this functional, namely $Y_{\mathcal{T}_0}(\mathbf{z})$, resembles a truncated version of a partition function\footnote{Here we prefer to use the terminology \emph{partition function} rather than \emph{generating function}, since the second terminology usually refers to \emph{formal} power series. Note that here, since the decomposition of maximal weight underlying the definition of the numbers $\omega(T,\mathbf{z})$ depends on $\mathbf{z}$, the quantity $\omega(T,\mathbf{z})$ is only \emph{piecewise} polynomial in the $z^U$, $U\in \mathcal{U}_0$, so that $Y_{\mathcal{T}_0}(\mathbf{z})$ is \emph{not} a formal power series in $\mathbf{z}$.} of trees weighted by their maximal weight. In this section we are going to use this fact to prove that some other functional of the $z^U_{n,\alpha}$, which we are directly interested in to prove Proposition~\ref{prop:AvsB}, is bounded.

\medskip

In this section $\cU_0$ and $\epsilon >0$ are \emph{fixed}.

\subsection{Partition functions of rooted and unrooted trees}
\label{subsec:defspartition}

In this subsection we consider an infinite version of the partition function appearing in the L.H.S. of \eqref{eq:sumBound}, where the summation on $\mathcal{T}_0$ is replaced by a summation on the set of all rooted trees $\mathcal{T}$. We are going to show that if this partition function is finite, then the unrooted version of this partition function is at most $\frac{1}{2}$ (Lemma~\ref{lemma:disymmetry}).

Recall that for any rooted tree $T$, and $\mathbf{z}\in \Lambda$, we defined $\omega(T,\mathbf{z})$ as the maximum weight of an $\cU_0$-admissible decomposition of $T$.
For any $\mathbf{z} \in \Lambda$, we let $Y(\mathbf{z}) \in \mathbb{R}_+\cup \{\infty\}$ be defined by the following infinite sum:
$$
Y(\mathbf{z}) := \sum_{T\in \cT} \frac{\omega(T,\mathbf{z})}{\mathrm{Aut}_r(T)},
$$
where the sum is taken over \emph{all} rooted (unlabeled) trees. Note that, by double-counting, this sum is also equal to the following sum, taken on all rooted \emph{labeled} trees:
$$
Y(\mathbf{z}) = \sum_{T\in \cT^\ell} \frac{\omega(T,\mathbf{z})}{|T|!}.
$$
In words, $Y(\mathbf{z})$ is the exponential partition function of all rooted labeled trees, counted with their maximum weight. We let $D$ be the domain of convergence of this sum:
$$
D:=\{\mathbf{z} \in \Lambda, Y(\mathbf{z})<\infty\}.
$$
\begin{lemma}
$D$ is closed downwards for the product order (i.e. for any $\mathbf{z},\mathbf{z'}\in \Lambda$ such that $(z')^U\leq z^U$ for every $U\in \cU_0$, if $\mathbf{z}\in D$ then $\mathbf{z}'\in D$) and $D$ is bounded.
\end{lemma}
\begin{proof}
The first assumption is straightforward, so to prove the second one it is enough to see that for each $U\in \cU_0$ we have $Y(\mathbf{z})=\infty$ where $\mathbf{z}$ is zero everywhere except for $z^U=|U|!$. We can construct a labeled tree of size $n|U|$ by attaching successively $n$ copies of $U$ by edges. The number of distinct ways to do that is equal to the number of rooted labeled trees of size $n$, which is $n^{n-1}$, times the number of ways to distribute the labels in the different copies of $U$, which is at least $\frac{1}{n!}\binom{n|U|}{|U|,\dots,|U|}= \frac{(n|U|)!}{n!(|U|!)^n}$.
Of course we do not obtain all trees of size $n|U|$ with this construction, but this is enough to obtain the lower bound:
$$
Y(\mathbf{z})\geq \sum_{n\geq 1} \frac{\frac{1}{n!}\binom{n|U|}{|U|,\dots,|U|} n^{n-1}}{(n|U|)!} (z^U)^n =
\sum_{n\geq 1} \frac{n^{n-1}}{n!(|U|!)^n} (z^U)^n = \sum_{n\geq 1} \frac{n^{n-1}}{n!}\;.
$$
The last sum is divergent, which concludes the proof of the claim.
\end{proof}

We note that $\omega(T,\mathbf{z})$ does not depend on the root of $T$, so this quantity is well defined for \emph{unrooted} trees $U\in \cU$. We can thus introduce the ``unrooted version'' of the partition function $Y$:
$$
Y^u(\mathbf{z}) := \sum_{U\in \cU} \frac{\omega(U,\mathbf{z})}{\mathrm{Aut}_u(U)}
= \sum_{U\in \cU^\ell} \frac{\omega(U,\mathbf{z})}{|U|!}.
$$
Note that $Y^u(\mathbf{z})$ is also given by the following expressions:
$$
Y^u(\mathbf{z})
= \sum_{T\in \cT^\ell} \frac{\omega(T,\mathbf{z})}{|T|\cdot|T|!}
= \sum_{T\in \cT} \frac{\omega(T,\mathbf{z})}{|T|\cdot \mathrm	{Aut}_r(T)}.
$$
It is clear that $Y^u(\mathbf{z})\leq Y(\mathbf{z})$ and in particular $Y^u(\mathbf{z}) <\infty$ if $\mathbf{z} \in D$.

The following statement, which is a variant in our context of the celebrated \emph{dissymmetry theorem} (see~\cite{quebecois}), is where the constant $\frac{1}{2}$ from our main theorem (Theorem~\ref{thm:main}) appears:
\begin{lemma}[Supermultiplicative dissymmetry theorem]\label{lemma:disymmetry}
If $\mathbf{z} \in D$, then the rooted series and unrooted series are related by the following inequality:
\begin{eqnarray}\label{eq:disymmetry}
Y(\mathbf{z})- Y^u(\mathbf{z}) \geq \frac{1}{2} Y(\mathbf{z})^2.
\end{eqnarray}
In particular for all $\mathbf{z} \in D$ one has
\begin{eqnarray}\label{eq:magicbound}
Y^u(\mathbf{z})\leq \frac{1}{2}.
\end{eqnarray}
\end{lemma}
\begin{proof}
Let $U\in \cU^{\ell}$ be a labeled unrooted tree. Then the number $e(U)=|U|-1$ of edges of $U$ and the number $v(U)=|U|$ of vertices of $U$ are related by the equation:
$$
v(U)-1 = e(U)\;.
$$
By multiplying this equality by $\omega(U,\mathbf{z})/|U|!$ and summing over all unrooted labeled trees $U$, it follows that the quantity $Y^e(\mathbf{z}):=Y(\mathbf{z})-Y^u(\mathbf{z})$ can be interpreted as the exponential partition function of all labeled trees \emph{with one marked edge}, counted with their maximum weight.
Now let $U$ be a labeled tree with a marked edge $e$. Removing $e$ splits $U$ into two connected components $T_1,T_2\ \in \cT^{\ell}$ that are naturally rooted at a vertex, and by definition of the maximum weight we have the supermultiplicativity property:
$$
\omega(U,\mathbf{z}) \geq \omega(T_1,\mathbf{z}) \cdot \omega(T_2,\mathbf{z}).
$$
Indeed, the right hand side is the weight of the $\cU_0$-admissible decomposition of $U$ induced by the decomposition with maximum weight of each of its components, and the weight of this decomposition is a lower bound on the maximum weight. Conversely, given any two rooted labeled trees $T_1$ and $T_2$ whose sizes add up to $|U|$, there are $\frac{|U|!}{|T_1|!|T_2|!}$ ways to distribute the labels in $[1..|U|]$ between them to build a labeled tree $U$ of this form, and each tree $U$ with a marked edge is obtained in exactly two ways by this construction.
Since all sums are absolutely convergent, we thus get that
$$Y^e(\mathbf{z})\geq \frac{1}{2} Y(\mathbf{z})^2,$$
which gives \eqref{eq:disymmetry}.

The bound \eqref{eq:magicbound} follows since by definition of $D$, $Y(\mathbf{z})$ is a well defined real number, and since for all $y\in \mathbb{R}$ one has $y-\frac{1}{2}y^2 \leq \frac{1}{2}$.
\end{proof}

\begin{rem}
The partition function $T(x)$ of all rooted trees, which is solution of the equation $T(x) = x \exp(T(x))$, has radius of convergence $e^{-1}$, and its value at the dominant singularity is $T(e^{-1})=1$. Moreover, it is classical  that the generating function of \emph{unrooted} trees is given by $T^u(x)=T(x)-\frac{1}{2}T(x)^2$ (see for example~\cite{quebecois}). It follows that, at the dominant singularity, one has $T^u(e^{-1})=\frac{1}{2}$. Note that this implies~\eqref{eq:disymmetry} in the case where $\cU_0$ is a singleton, and that this also shows that \eqref{eq:disymmetry} is tight. We also note that, using classical singularity analysis~\cite{flajolet}, this enables one to reprove the result of R\'enyi~\cite{Renyi} that says that a random forest of size $n$ is connected with probability $e^{-1/2} + o(1)$ when $n$ tends to infinity.
\end{rem}

The last partition functions we define are the functionals $Y^u_{\mathcal{U}_0}(\mathbf{z})$ and $\widetilde{Y}^u_{\mathcal{U}_0}(\mathbf{z})$, defined by:
\begin{align*}
      {Y}^u_{\mathcal{U}_0}(\mathbf{z}) := \sum_{U \in \cU_0} \frac{\omega(U,\mathbf{z})}{\mathrm{\Aut_u}(U)} \ \ ,\ \  
\widetilde{Y}^u_{\mathcal{U}_0}(\mathbf{z}) := \sum_{U \in \cU_0} \frac{z^U}{\mathrm{\Aut_u}(U)}.
\end{align*}
Note that the sums are taken over all the elements of $\cU_0$, that are considered as \emph{unrooted} trees. 

\subsection{Optimization}
\label{subsec:optimization}

In the last subsection we have shown (Lemma~\ref{lemma:disymmetry}) that the fact that $Y(\mathbf{z})<\infty$ implies that $Y^u(\mathbf{z})\leq \frac{1}{2}$. The goal of this subsection, achieved in the next proposition, is to transfer this result to truncated analogues of these partition functions.

For all $k\geq 1$, define the following truncated version of $Y(\mathbf{z})$:
$$
Y_{\leq k}(\mathbf{z}):= \sum_{{T\in \cT} \atop {|T|\leq k}} \frac{\omega(T,\mathbf{z})}{\mathrm{Aut}_r(T)}.
$$
Note that $Y_{\leq k}(\mathbf{z})$ is defined by a finite sum, hence it is a well defined real number for \emph{all} $\mathbf{z}\in \Lambda$. We also define $Y_{=k}(\mathbf{z})$ to be the contribution of trees of size exactly $k$ to $Y(\mathbf{z})$:
$$
Y_{=k}(\mathbf{z}):=
Y_{\leq k}(\mathbf{z})
-
Y_{\leq k-1}(\mathbf{z}).
$$

\begin{prop}\label{prop:obj}
There exists a $k_*$, depending only on $\epsilon$ and $\mathcal{U}_0$,  such that for every $\mathbf{z}\in \Lambda$ satisfying $Y_{\leq k_*}(\mathbf{z}) \leq 1.5$, we have
$$
\widetilde{Y}^u_{\mathcal{U}_0}(\mathbf{z}) \leq \frac{1}{2}(1+\epsilon).
$$
\end{prop}
\begin{rem}
If $k\geq k_*$ then $Y_{\leq k_*}(\mathbf{z}) \leq Y_{\leq k}(\mathbf{z})$. Therefore if necessary the integer $k_*$ can be replaced by any larger value without changing the conclusion of the proposition.
\end{rem}
\begin{rem}
The constant $1.5$ in the above proposition could be replaced by any constant larger than $1$ (as the proof will show). To keep the notation light we preferred to fix some arbitrary value that is good enough for our proof.
\end{rem}

Note that for any $\mathbf{z}\in\Lambda$, if we define $\mathbf{z}_*\in \Lambda$ by the fact  that for all $U \in \mathcal{U}_0$ we have
$
\mathbf{z}_*^U = \omega(U,\mathbf{z}),
$
then $\omega(T,\mathbf{z})= \omega(T,\mathbf{z}_*)$ for any tree $T\in \mathcal{T}$ (this is easily seen by considering maximum weight decompositions). Since it is always true that $z^U \leq \omega(U,\mathbf{z})$, if follows that replacing $\mathbf{z}$ by $\mathbf{z}_*$ does not change the value of $Y_{\leq k}(\mathbf{z})$, while only making $\widetilde{Y}^u_{\mathcal{U}_0}(\mathbf{z})$ larger or equal. Therefore
\begin{align}
\max\{\widetilde{Y}^u_{\mathcal{U}_0}(\mathbf{z}): \ Y_{\leq k}(\mathbf{z}) \leq 1.5\}
&=
\max\{\widetilde{Y}^u_{\mathcal{U}_0}(\mathbf{z}): \ Y_{\leq k}(\mathbf{z}) \leq 1.5 \text{ and } \forall U\in \mathcal{U}_0,\ z^U=\omega(U,\mathbf{z})\ \}
\nonumber\\
&=
\max\{Y^u_{\mathcal{U}_0}(\mathbf{z}): \ Y_{\leq k}(\mathbf{z}) \leq 1.5 \text{ and } \forall U\in \mathcal{U}_0,\ z^U=\omega(U,\mathbf{z})\ \}.
\label{eq:equalmax}
\end{align}

Now, let us fix a sequence $(\mathbf{z}_k)_{k\geq u_{max}}$, such that for all $k\geq u_{max}$ we have:
$$
\mathbf{z}_k\in \{ \mathbf{z}: \ Y_{\leq k}(\mathbf{z}) \leq 1.5 \text{ and } \forall U\in \mathcal{U}_0,\  z^U=\omega(U,\mathbf{z})\}.
$$
Note that the set from where $\mathbf{z}_k$ is selected, is bounded: for $k\geq u_{\max}$, we have $Y_{\leq k}(\mathbf{z})\geq \omega(T,\mathbf{z})/\Aut_r(T)$ for each rooted tree $T$ that, as an unrooted tree, belongs to $\cU_0$. This directly implies that the sequence $\mathbf{z}_{k}$ is uniformly bounded.
We can thus extract an increasing  sequence $k_i$ such that the corresponding subsequence converges, and we note:
$$
\mathbf{z}_\infty := \lim_{i \rightarrow \infty} \mathbf{z}_{k_i}.
$$
Our first step in the proof of the proposition is the following lemma.
\begin{lemma}\label{lemma:zInfty}
The point $\mathbf{z}_\infty$ belongs to $\bar{D}$ (the closure of $D$).
\end{lemma}
\begin{proof}
We go by contradiction. Suppose $\mathbf{z}_\infty$ does not belong to $\bar{D}$. Then there exists $\delta>0$ such that $\frac{1}{1+\delta} \circ \mathbf{z}_\infty \not\in D$, where we use the notation $\circ$ for the \emph{scaled multiplication} of a vector $\mathbf{z}\in\Lambda$ by a scalar $\lambda\in \mathbb{R}$:
$$
(\lambda \circ \mathbf{z})^U = \lambda^{|U|}z^U, \ \ U \in \cU_0.
$$
We thus have $Y(\frac{1}{1+\delta} \circ \mathbf{z}_\infty)=\infty$, \textit{i.e.}
$\sum_{\ell\geq 1} Y_{=\ell}(\mathbf{z}_{\infty}) (1+\delta)^{-\ell} = \infty$.
This shows that there exists an infinite sequence $(\ell_j)_{j\geq 1}$ tending to infinity such that for all $j\geq 1$ one has:
$$
Y_{=\ell_j}(\mathbf{z}_{\infty}) \geq \left(1+\delta/2\right)^{\ell_j}.
$$
Now we claim that there exists some $i_0$ such that for $i\geq i_0$ one has, for every rooted tree $T\in \cT$:
\begin{eqnarray}\label{eq:claim}
\omega(T,\mathbf{z}_{k_i}) \geq \omega(T,\mathbf{z}_\infty)\left(1-\frac{\delta}{4}\right)^{|T|}.
\end{eqnarray}
If we admit this claim, we can conclude the proof as follows. We have, for $i\geq i_0$ and $j\geq 1$:
\begin{eqnarray*}
Y_{=\ell_j}(\mathbf{z}_{k_i}) &\geq&
Y_{=\ell_j}(\mathbf{z}_{\infty}) \left(1-\delta/4\right)^{\ell_j}\\
&\geq&
\left((1-\delta/4)(1+\delta/2)\right)^{\ell_j}.
\end{eqnarray*}
But $(1-\delta/4)(1+\delta/2)$ is larger than $1$ provided we took $\delta$ small enough (and we can do that), so there exists some $j$ such that $\big((1-\delta/4)(1+\delta/2)\big)^{\ell_j} >1.6$, which in turns implies that for $i\geq i_0$, one has $Y_{=\ell_j}(\mathbf{z}_{k_i})>1.6$. Now we can choose $i$ large enough so that $k_i \geq \ell_j$, and we get that:
$$
Y_{\leq k_i}(\mathbf{z}_{k_i}) \geq Y_{=\ell_j}(\mathbf{z}_{k_i})>1.6,
$$
which is a contradiction.

So it just remains to prove the claim in~\eqref{eq:claim}. We let $I\subset \cU_0$ denote the set indexing coordinates of $\mathbf{z}_\infty$ that are equal to zero, \textit{i.e.}:
$$
I:=\{U\in \cU_0, \ (z_\infty)^U=0\}.
$$
Since $\mathbf{z}_{k_i}$ converges to $\mathbf{z}_\infty$, and since $(\mathbf{z}_\infty)^U\neq 0$ for $U\in\cU_0\setminus I$, each of the ratios $\frac{(z_{k_i})^U}{(z_{\infty})^U}$ converge to $1$ when $i$ tends to infinity, for $U \in \cU_0\setminus I$. Therefore there exists $i_0$ such that for $i\geq i_0$, we have for all $U\in\cU_0\setminus I$:
\begin{eqnarray*}
\left|\left(\frac{(z_{k_i})^U}{(z_{\infty})^U}\right)^{1/|U|}-1 \right|\;
<\delta/4.
\end{eqnarray*}
We can now prove the claim~\eqref{eq:claim}.
%
First, if $\omega(T,\mathbf{z}_\infty)=0$ then the claim is obviously true. Otherwise, consider an $\cU_0$-admissible decomposition $\mathbf{T}$ of $T$ that gives rise to the maximum weight $\omega(T,\mathbf{z}_\infty)$. Since $\omega(T,\mathbf{z}_\infty)\neq 0$, the decomposition only uses unrooted trees in $\cU_0\setminus I$. 
We thus have:
\begin{align}\label{eq:compareDecompositions}
\frac{\omega(\mathbf{T},\mathbf{z}_{k_i})}{
\omega(\mathbf{T},\mathbf{z}_{\infty})  }=
\prod_{U\in \cU_0\setminus I} \left(\frac{(z_{k_i})^U}{(z_{\infty})^U} \right)^{\nu(U)},
\end{align}
where $\nu(U)$ is the number of times $U$ is used in the decomposition $\mathbf{T}$. Since $\sum_{U\in \cU_0} |U| \nu(U)=|T|$, the ratio~\eqref{eq:compareDecompositions} is larger than $(1-\delta/4)^{|T|}$, and the claim follows since $\omega(T,\mathbf{z}_{k_i}) \geq \omega(\mathbf{T},\mathbf{z}_{k_i})$.
\end{proof}

We can now prove the proposition:
\begin{proof}[Proof of Proposition~\ref{prop:obj}]
For every $k\geq u_{\max}$, we now choose
$$
\mathbf{z}_k\in \mathop{argmax}_{\mathbf{z}\in \Lambda} \{Y^u_{\mathcal{U}_0}(\mathbf{z}): \ Y_{\leq k}(\mathbf{z}) \leq 1.5 \text{ and } \forall U\in \mathcal{U}_0,\  z^U=\omega(U,\mathbf{z})\},
$$
and we select $\mathbf{z}_{\infty}$ as before.

From Lemma~\ref{lemma:zInfty} and the fact that $D$ is closed downwards for the product order, for all $\delta>0$ we have $\frac{1}{1+\delta}\circ \mathbf{z}_{\infty} \in D$, where we recycle the notation $\circ$ for the scaled product from the previous proof. From Lemma~\ref{lemma:disymmetry}, it follows that $Y^u( \frac{1}{1+\delta}\circ \mathbf{z}_{\infty}) \leq \frac{1}{2}$. Now, recall that we restricted to $\mathbf{z}\in \Lambda$ such that $z^U=\omega(\mathbf{z},U)$ for every $U\in\cU_0$. This implies that
$$
Y^u\left( \frac{1}{1+\delta} \circ \mathbf{z}_\infty\right) \geq
Y^u_{\mathcal{U}_0} \left( \frac{1}{1+\delta} \circ \mathbf{z}_\infty\right).
$$
Moreover, $Y^u_{\mathcal{U}_0}\left( \frac{1}{1+\delta} \circ \mathbf{z}_\infty\right) \geq (1+\delta)^{-u_{\max}} Y^u_{\mathcal{U}_0}(\mathbf{z}_\infty)$, so
$Y^u_{\mathcal{U}_0}(\mathbf{z}_\infty)\leq (1+\delta)^{u_{\max}} \cdot \frac{1}{2}$. Since this is true for any $\delta>0$, we obtain that
$Y^u_{\mathcal{U}_0}(\mathbf{z}_\infty)\leq\frac{1}{2}.
$
Since $\mathbf{z}_{k_i}$ converges to $\mathbf{z}_\infty$ and $Y^u_{\mathcal{U}_0}(\mathbf{z})$ is clearly continuous, we can choose $i$ large enough and $k_*=k_i$ so that $Y^u_{\mathcal{U}_0}(\mathbf{z}_{k_*})\leq\frac{1}{2}(1+\epsilon)$.
By the choice of $\mathbf{z}_{k}$ and by~\eqref{eq:equalmax}, it follows that 
$\max\{\widetilde{Y}^u_{\mathcal{U}_0}(\mathbf{z}): \ Y_{\leq k_*}(\mathbf{z}) \leq 1.5\}\leq \frac{1}{2}(1+\epsilon)$, which implies the proposition.
\end{proof}

\section{Finishing the proof}
\label{sec:finishproof}

In this section we conclude the proof of Proposition~\ref{prop:AvsB} (hence of the main theorem). The idea of the proof is to combine the main results of Section~\ref{sec:local} (Corollary~\ref{cor:sumBound}) and of Section~\ref{sec:partitionfunctions} (Proposition~\ref{prop:obj}) and to apply them to a well chosen set of boxes.

\subsection{Boxing lemma}
\label{subsec:boxing}

The results of the previous section give us bounds on the variables $\mathbf{z}_{n,\alpha}$ defined by~\eqref{eq:defzn}, which gives us some control on the ratio of the sizes of the sets $\mathcal{A}_{n,[\alpha]^w_{q_*}}$ and $\mathcal{B}^U_{n,[\alpha]^w}$, where $[\alpha]^w$ is some box inside the parameter space $\mathcal{E}$. In order to use this information in the next subsection, we first show that there exists a partition of the parameter space $\cE$ into disjoint boxes $[\beta_i]^w$ such that they are $2q_*$-apart and they capture most of the graphs in $\mathcal{B}_n^U$ for each $U\in\mathcal{U}_0$.
\begin{lemma}[Box partitioning lemma] \label{lemma:grid}
For all $\epsilon>0$, $\cU_0$ and $\cT_0$, there exist $q_*$, $w$ and $n_0$ such that for all $n\geq n_0$ the following is true:
There exist $K$ and a family of boxes $[\beta_i]^w \subset \cE$ of size $K$ such that
\begin{enumerate}[label=(\subscript{P}{\arabic*})]
\item
The $q_*$-neighbourhoods of boxes form a partition  of $\cE$; i.e.
$$\displaystyle \cE = \biguplus_{i=1}^K [\beta_i]^w_{q_*},$$
\item
Boxes capture a large fraction of each set $\mathcal{B}^U_n$; i.e. for each $U\in \cU_0$, we have:
$$
\sum_{i=1}^K \sum_{\beta \in [\beta_i]^w} |\mathcal{B}^U_{n,\beta}| \geq  (1-\epsilon) |\mathcal{B}_n^{U}|.
$$
\end{enumerate}
\end{lemma}
\begin{proof}
Let $d=|\cT_0|$ and let $q_*$ be chosen as in Lemma~\ref{lemma:localSwitching}. Choose $w\geq 0$ and $n_0$ such that for every $n\geq n_0$
$$
 1- \left(1-\frac{w+2q_*}{n} \right)^d \left(1+\frac{2q_*}{w} \right)^{-d} \leq \epsilon |\cU_0|^{-1}\;.
$$
Consider the following set
$$
\Gamma_0=\{\beta\in \cE:\; \exists T\in\cT_0,\exists j\geq 1, j(w+2q_*)-2q_*\leq \beta(T)< j(w+2q_*)\}\;.
$$
The set $\Gamma_0$ can be seen as a subset of hyperplanes of width $2q_*$ equally spaced at distance $w$ in each direction. Observe that $\Gamma_0^c=\cE \setminus \Gamma_0$ contains a set of $K= \left(\frac{n}{w+2q_*}-1\right)^d$ boxes of width $w$ that are $(2q_*)$-apart. However, these boxes might not contain a negligible part of the graphs in $\mathcal{B}_n$.

The size of the $\Gamma_0^c$ satisfies,
\begin{align*}
|\Gamma_0^c| &\geq w^d  \left(\frac{n}{w+2q_*}-1\right)^d \geq \left(1-\frac{w+2q_*}{n} \right)^d \left(1+\frac{2q_*}{w} \right)^{-d} n^d\;.
\end{align*}
For the remaining of the proof, let us consider $\Gamma_0 \subseteq \mathbb{Z}^d/ \cE$. Choose $\beta$ uniformly at random from $\mathbb{Z}^d/ \cE$. We write $\Gamma_0+\beta=\{\gamma+\beta :\;\gamma\in \Gamma_0\}$ to denote the translation of the set $\Gamma_0$ by $\beta$. Recall that $|\cE|=n^d$.  Then
\begin{align}\label{eq:size}
|\Gamma_0+\beta|=|\Gamma_0|\leq \left( 1- \left(1-\frac{w+2q_*}{n} \right)^d \left(1+\frac{2q_*}{w} \right)^{-d} \right)n^d \leq \epsilon |\cU_0|^{-1} n^d\;.
\end{align}
We now define the following measure $\mu$ over the set $\cE$. For each $U\in \cU_0$ and $\Gamma\subseteq \cE
$, let $b^U_\Gamma=\frac{|\mathcal{B}^U_{n,\Gamma}|}{|\cU_0||\mathcal{B}^U_n|}$. For every $\Gamma \subseteq \cE
$, we define
$$
\mu(\Gamma)= \sum_{U\in \cU_0} b^U_{\Gamma}\;.
$$
Observe that if $\mu(\Gamma)\leq \epsilon|\cU_0|^{-1}$, then, for every $U\in \cU_0$
$$
|\mathcal{B}^U_{n,\Gamma}|=b^U_\Gamma|\cU_0||\mathcal{B}^U_n|\leq \mu(\Gamma)|\cU_0||\mathcal{B}^U_n| \leq  \epsilon |\mathcal{B}^U_n|\;.
$$
If $\beta\in \cE$ is chosen uniformly at random and using~\eqref{eq:size}, we have
$$
\mathbb{E_\beta} [\mu(\Gamma_0+\beta)] = \sum_{\gamma\in \cE
}\mu(\gamma)\Pr(\gamma\in \Gamma_0+\beta) = \sum_{\gamma\in \cE
} \sum_{U\in \cU_0} b^U_{\gamma} \cdot\frac{|\Gamma_0|}{n^d}\leq \epsilon  |\cU_0|^{-1}\;,
$$
where we used that $\sum_{\gamma\in \cE}\sum_{U\in\cU_0} b^U_\gamma = 1$.

Thus, there exists a $\beta_0\in \cE$ such that  $\mu(\Gamma_0+\beta_0)\leq \epsilon  |\cU_0|^{-1}$. Then, the set $\Gamma_1=\mathcal{E}\setminus(\Gamma_0+\beta_0)$ is a set of $(2q_*)$-apart boxes $[\beta_i]^w$ for $i\in [1..K]$ that satisfies
$$
\sum_{i=1}^K \sum_{\beta \in [\beta_i]^w} |\mathcal{B}^U_{n,\beta}|= |\mathcal{B}^U_{n,\Gamma_1}| = |\mathcal{B}^U_{n}| - |\mathcal{B}^U_{n,(\Gamma_0+\beta_0)}|\geq  (1-\epsilon) |\mathcal{B}^U_{n}|\;,
$$
for every $U\in \cU_0$.
\end{proof}

\subsection{Proof of Proposition~\ref{prop:AvsB}}
\label{subsec:endofproof}

Let $\epsilon'>0$ be fixed, and let $\epsilon:=\epsilon'/3$. Let us choose $\cU_0$ to be the family of all unrooted unlabeled trees of size at most $u_{\max}$, where $u_{max}$ is chosen such that  $(u_{\max})^{-1}<\epsilon$. Now that $\cU_0$ is fixed, we can apply Proposition~\ref{prop:obj}, with our current value of $\epsilon$, and we let $k_*$ be the value given by this proposition. The integer $k_*$ depends on $\epsilon$ (and also on $\cU_0$, that itself depends on $\epsilon$). We now let $\cT_0$ be the set of all rooted trees of size at most $k_*$.

Now that $\epsilon$, $\cU_0$, and $\cT_0$ are fixed, so is the constant $q_*$ given by Lemma~\ref{lemma:localSwitching}.
We can then apply Lemma~\ref{lemma:grid} to get some constants $w$ and $n_0$, such that for every $n\geq n_0$ there exists a family of $K$ boxes $[\beta_i]$ satisfying $(P_1)$ and $(P_2)$. All these constants depend on $\epsilon$ (and also on $\cU_0$ and $\cT_0$, that both also depend on $\epsilon$).

We can now choose $n_1\geq n_0$ large enough, so that
\begin{eqnarray}\label{eq:chooseN}
 (w+q_*)(2k_*)^{k_*-1} |\cT_0| < 0.4 n_1.
\end{eqnarray}
Note that the left-hand side is the quantity $C$ that appears in \eqref{eq:sumBound}, with $t_{\max}=k_*$.
For a further use, we will also assume that $n_1\geq u_{\max}/\epsilon$.

For $n\geq n_1$, let $[\alpha]^w$~(where $\alpha=\beta_i$, for some $1\leq i\leq K$) be one of the boxes given by the box partitioning lemma (Lemma~\ref{lemma:grid}). Recall the definition of $\mathbf{z}_{n,\alpha}=(z^U_{n,\alpha})_{U\in \cU_0}$ given in \eqref{eq:defzn}:
$$
z_{n,\alpha}^U:=\Aut_u(U) \frac{|\mathcal{B}^U_{n,[\alpha]^w}|}{|\mathcal{A}_{n,[\alpha]^w_{q_*}}|} \left(1-\frac{|U|}{n}\right).
$$
Then we have by Corollary~\ref{cor:sumBound} and the bound~\eqref{eq:chooseN} that:
$$
Y_{\leq k_*}(\mathbf{z}_{n,\alpha}) =\sum_{T\in\cT_0} \frac{\omega(T, \mathbf{z}_{n,\alpha})}{\Aut_r(T)}\leq 1.4\;.
$$
From Proposition~\ref{prop:obj}, this implies that:
$$
\sum_{U\in\cU_0} \frac{|\mathcal{B}^U_{n,[\alpha]^w}|}{|\mathcal{A}_{n,[\alpha]^w_{q_*}}|} \left(1-\frac{|U|}{n}\right)
= \sum_{U\in\cU_0} \frac{z^U_{n,\alpha}}{\Aut_u(U)}
=\widetilde{Y}^u_{\mathcal{U}_0}(\mathbf{z}_{n,\alpha})\leq \frac{1}{2} \left(1+\epsilon \right).
$$
Since $|U|/n\leq u_{\max}/n <\epsilon$ for all $U\in \cU_0$, we deduce that we have:
\begin{align}\label{eq:next_paper1}
\sum_{U\in\cU_0}
\left|\mathcal{B}^U_{n,[\alpha]^w}\right| \leq \frac{1}{2}  \left|\mathcal{A}_{n,[\alpha]^w_{q_*}}\right| \left(1+\epsilon \right)(1-\epsilon)^{-1}.
\end{align}
We now sum the last inequality over all the boxes $[\beta_i]^w$ provided by  Lemma~\ref{lemma:grid}. We obtain:
\begin{align*}
\sum_{U\in\cU_0} (1-\epsilon)|\mathcal{B}^U_{n}|&\leq  \sum_{U\in\cU_0} \sum_{i=1}^{K} |\mathcal{B}^U_{n,[\beta_i]^w}|\leq \frac{1}{2} \left(1+\epsilon \right)(1-\epsilon)^{-1}  \sum_{i=1}^{K} |\mathcal{A}_{n,[\beta_i]^w_{q_*}}| \leq  \frac{1}{2}  |\mathcal{A}_{n}|  \left(1+\epsilon \right) (1-\epsilon)^{-1}.
\end{align*}
Here: the central inequality is the summation of the previous bound;
the leftmost inequality comes from Property $(P_2)$ of Lemma~\ref{lemma:grid} (boxes capture most of the mass of the sets $\mathcal{B}_n^U$); the rightmost inequality comes from Property $(P_1)$ of Lemma~\ref{lemma:grid} (boxes are $(2q_*)$-apart, so the sets $[\beta_i]^w_{q_*}$ are disjoint).
We have just proved that for $n\geq n_1$,
$$
\sum_{U\in\cU_0}
|\mathcal{B}^U_{n}| \leq \frac{1}{2} \left(1+\epsilon \right)(1-\epsilon)^{-2} |\mathcal{A}_n|.
$$
Again, using the same simple double counting argument as we used in the proof of Lemma~\ref{lemma:simpleCounting}, we have:
$$
u_{\max}(n-u_{\max}) \sum_{U\in\cU\setminus\cU_0}
|\mathcal{B}^U_{n}| \leq n |\mathcal{A}_n|.
$$
By assumption, $(u_{\max})^{-1}<\epsilon$, so $\frac{n}{u_{\max}(n-u_{\max})} = \frac{1}{u_{\max}(1-u_{\max}/n)} \leq 2\epsilon$ if $n$ is large enough.
Therefore, for $n$ large enough, we have:
\begin{align}\label{eq:next_paper2}
\sum_{U\in\cU\setminus\cU_0}
|\mathcal{B}^U_{n}| \leq 2 \epsilon |\mathcal{A}_n|.
\end{align}
Putting all bounds together, we obtain:
$$
|\mathcal{B}_n|
= \sum_{U\in \cU} |\mathcal{B}^U_{n}|
= \sum_{U\in\cU_0} |\mathcal{B}^U_{n}|
+ \sum_{U\in\cU\setminus\cU_0} |\mathcal{B}^U_{n}|
\leq \left( \left(1+\epsilon \right)(1-\epsilon)^{-2}+ 2 \epsilon\right) \frac{1}{2} |\mathcal{A}_n|.
$$
Now, when $\epsilon'$ is small enough and $\epsilon:=\epsilon'/3$, we have $ \left(1+\epsilon \right)(1-\epsilon)^{-2}+ 2 \epsilon\leq (1+2 \epsilon')$. We thus get:
$$
|\mathcal{B}_n| \leq \left( \frac{1}{2}+\epsilon'\right) |\mathcal{A}_n|,
$$
which concludes the proof of Proposition~\ref{prop:AvsB}.

\subsection*{Acknowledgements}
We thank Colin McDiarmid and Sergey Norin for comments on an early version of this paper, and Mihyun Kang for interesting feedback. We also thank Colin McDiarmid for detecting a technical flaw in a previous version of the proof of Proposition~\ref{prop:ivsi1}.
We thank the organizers and participants of the workshops \textit{Graph Theory} and \textit{Combinatorics, Probability and Geometry} held in March--April 2015 in Bellairs institute, McGill University, Barbados, where this result was announced. 

\bibliographystyle{alpha}
\bibliography{biblio}

\end{document}